\documentclass[a4paper,12pt]{article}

\usepackage[left=2cm,right=2cm, top=2cm,bottom=2cm,bindingoffset=0cm]{geometry}

\usepackage{verbatim}
\usepackage{amsmath}
\usepackage{amsthm}
\usepackage{amssymb}
\usepackage{delarray}
\usepackage{hyperref}
\usepackage{mathrsfs}

\newcommand{\al}{\alpha}
\newcommand{\be}{\beta}

\newcommand{\la}{\lambda}

\newcommand{\eps}{\varepsilon}

\newcommand{\iy}{\infty}

\theoremstyle{plain}

\newtheorem{thm}{Theorem}
\newtheorem{lem}{Lemma}

\newtheorem{cor}{Corollary}

\theoremstyle{definition}

\theoremstyle{remark}

\newtheorem{remark}{Remark}

 \allowdisplaybreaks

\begin{document}

\begin{center}
{\large\bf A 2-edge partial inverse problem for the Sturm-Liouville operators with singular potentials
on a star-shaped graph 
}
\\[0.2cm]
{\bf Natalia P. Bondarenko} \\[0.2cm]
\end{center}

\vspace{0.5cm}

{\bf Abstract.} Boundary value problems for Sturm-Liouville operators with potentials from the class $W_2^{-1}$ on a star-shaped graph are considered. We assume that the potentials are known on all the edges of the graph except two, and show that the potentials on the remaining edges can be constructed by fractional parts of two spectra. A uniqueness theorem is proved, and an algorithm for the constructive solution of the partial inverse problem is provided. The main ingredient of the proofs is the Riesz-basis property of specially constructed systems of functions.

\medskip

{\bf Keywords:} partial inverse problem, quantum graph, Sturm-Liouville operator, singular potential, Weyl function, Riesz basis.

\medskip

{\bf AMS Mathematics Subject Classification (2010):} 34A55 34B09 34B24 34B45 47E05 

\vspace{1cm}

{\large \bf 1. Introduction}

\bigskip

The paper concerns the theory of inverse spectral problems for differential operators on geometrical graphs. Differential operators on graphs (or so-called quantum graphs) have been actively studied by mathematicians in recent years and have applications in different branches of science and engineering (see \cite{Kuch02, PPP04} and the bibliography therein). Inverse spectral problems consist in recovering differential operators from their spectral characteristics. Nowadays inverse problems for quantum graphs attract much attention of mathematicians. The reader can find an extensive bibliography on this subject in the survey \cite{Yur16}.

In this paper, we consider a star-shaped graph $G$ with edges $e_j$, $j = \overline{1, m}$, of equal length $\pi$. 
For each edge $e_j$, introduce a parameter $x_j \in [0, \pi]$. The value $x_j = 0$ corresponds to the boundary vertex, associated with $e_j$, and $x_j = \pi$ corresponds to the internal vertex.

Let $y = [y_j(x_j)]_{j = 1}^m$ be a vector function on the graph $G$, and let $q_j$, $j = \overline{1, m}$, be real-valued functions from $W_2^{-1}(0, \pi)$, i.e. $q_j = \sigma_j'$, $\sigma_j \in L_2(0, \pi)$, where the derivative is considered in the sense of distributions. The functions $\sigma_j$ are called the {\it potentials}. The Sturm-Liouville operator
$$
\ell_j y_j := -y_j'' + q_j(x_j) y_j  
$$
on the edge $e_j$ can be understood in the following sense:
$$
\ell_j y_j = -(y_j^{[1]})' - \sigma_j(x_j) y_j^{[1]} - \sigma_j^2(x_j) y_j,
$$ 
where $y_j^{[1]} = y_j' - \sigma_j y_j$ is a {\it quasi-derivative}, and
$$
\mbox{Dom}(\ell_j) = \{ y_j \in W_2^1[0, \pi] \colon y_j^{[1]} \in W_1^1[0, \pi], \: \ell_j y_j \in L_2(0, \pi) \}. 
$$

Properties of Sturm-Liouville operators with singular potentials in the described form were established in \cite{SS99}. Inverse spectral problems on a {\it finite interval}, consisting in recovering singular potentials from different types of spectral characteristics, were extensively studied by R.O.~Hryniv and Ya.V.~Mykytyuk \cite{HM03, HM04-2spectra, HM04-half, HM04-transform}. 
However, as far as we know, there is the only paper \cite{FIY08}, concerning 
an inverse problem for Sturm-Liouville operators with the potentials $q_j$ from $W_2^{-1}$ on graphs. 

In the present paper, we study the system of the Sturm-Liouville equations on the graph~$G$:
\begin{equation} \label{eqv}
    (\ell_j y_j)(x_j) = \la y_j (x_j), \quad x_j \in (0, \pi), \: y_j \in \mbox{Dom}(\ell_j), \: j = \overline{1, m}.
\end{equation}
Let $L$ and $L_0$ be the boundary value problems for the system \eqref{eqv} with 
the standard matching conditions in the internal vertex
\begin{equation*} 
  	y_1(\pi) = y_j(\pi), \quad j = \overline{2, m}, 
    \quad \sum_{j = 1}^m y_j^{[1]}(\pi) = 0,
\end{equation*}
and the mixed boundary conditions
\begin{align*} 
L \colon \quad & y_j^{[1]}(0) = 0, \: j = \overline{1, p}, \quad y_j(0) = 0, \: j = \overline{p+1, m},\\
L_0 \colon \quad & y_j^{[1]}(0) = 0, \: j = \overline{1, p+1}, \quad y_j(0) = 0, \: j = \overline{p+2, m},
\end{align*}
where $2 \le p \le m -2$.

The asymptotic behavior of the spectrum of the problem $L$ is described by the following theorem, which can also be applied to the problem $L_0$. Everywhere below the same symbol $\{ \varkappa_n \}$ is used for different sequences from $l_2$.

\begin{thm} \label{thm:asympt}
The boundary value problem $L$ has a countable set of eigenvalues, which are real and can be numbered 
as $\{ \la_{nk} \}_{n \in \mathbb N, \, k = \overline{1, m} }$ (counting 
with their multiplicities) to satisfy the following asymptotic formulas
\begin{equation} \label{asymptrho}
\arraycolsep=1.4pt\def\arraystretch{2.2}
\left.
\begin{array}{ll}
    \rho_{n1} & =  n - 1 + \dfrac{\al}{\pi} + \varkappa_n, \\ 
    \rho_{n2} & = n - \dfrac{\al}{\pi} + \varkappa_n, \\
    \rho_{nk} & = n -\dfrac{1}{2} + \varkappa_n,  \quad k \in \mathcal I_3, \\ 
    \rho_{nk} & = n + \varkappa_n, \quad k \in \mathcal I_4,
\end{array}
\right\}
\end{equation}
where $\rho_{nk} = \sqrt{\la_{nk}}$, $\al = \arccos \sqrt{\frac{p}{m}}$, $\mathcal I_3$ and $\mathcal I_4$ are some fixed sets of indices, such that 
$\mathcal I_3 \cup \mathcal I_4 = \overline{3, m}$, $\mathcal I_3 \cap \mathcal I_4 = \varnothing$, $|\mathcal I_3| = p-1$, $|\mathcal I_4| = m-p-1$.
For definiteness, we assume that $3 \in \mathcal I_3$ and $4 \in \mathcal I_4$.
\end{thm}

Theorem~\ref{thm:asympt} can be proved similarly to \cite[Theorem~1]{Bond17-mixed}. 

In the papers \cite{Bond17-mixed, Bond17}, we started to investigate the so-called {\it partial inverse problems} on graphs. Our research was motivated by the paper \cite{Yang10} by C.-F. Yang, who has shown, that the (regular) potential of the Sturm-Liouville operator on one edge of the star-shaped graph is uniquely specified by a fractional part of the spectrum, if the potentials on all the other edges are given. 
In the papers \cite{Bond17-mixed, Bond17}, we have developed a constructive method for the solution of such 1-edge partial inverse problems. The method is based on the Riesz-basis property of some systems of vector functions, and allows one to establish the local solvability of the partial inverse problems and the stability for their solutions. Note that the partial inverse problems on graphs generalize the Hochstadt-Lieberman problem on a finite interval \cite{HM04-half, HL78}.

In this paper, we demonstrate that the approach of \cite{Bond17-mixed, Bond17} can be applied to operators with singular potentials. Moreover, in contrast to the previous papers, we study a 2-edge inverse problem, when the potentials on two edges are unknown. In this case, one spectrum is not sufficient for recovering the both potentials, so we use a part of the spectrum of the boundary value problem $L$ and a part of the spectrum of $L_0$. We prove the uniqueness theorem and provide a constructive algorithm for the solution of the 2-edge inverse problem. The most challenging part of the research is the analysis of the Riesz-basicity for special systems of functions (see Appendix A).

Let us proceed to the problem formulation. Denote by $C_j(x_j, \la)$, $j = \overline{1,p + 1}$,  and $S_j(x_j, \la)$, $j = \overline{p+1, m}$ the solutions of equations \eqref{eqv}
under the initial conditions
\begin{equation} \label{init}
  	C_j(0, \la) = 1, \: C_j^{[1]}(0, \la) = 0, \quad S_j(0, \la) = 0, \: S_j^{[1]}(0, \la) = 1.
\end{equation}

Consider a sequence $\{ \la_{nk} \}_{n \in \mathbb N\, k = \overline{1, 4}}$ of eigenvalues of the problem $L$, satisfying \eqref{asymptrho}, and a sequence $\{ \mu_{nk} \}_{n \in \mathbb N, \, k = 1, 2}$ of eigenvalues of the problem $L_0$, satisfying the following asymptotic relations
\begin{equation} \label{asymptmu}
\arraycolsep=1.4pt\def\arraystretch{2.2}
\left.
\begin{array}{ll}
    \sqrt{\mu_{n1}} & =  n - 1 + \dfrac{\al_1}{\pi} + \varkappa_n, \\ 
    \sqrt{\mu_{n2}} & = n - \dfrac{\al_1}{\pi} + \varkappa_n, \\
\end{array}
\right\}
\end{equation}
where $\al_1 = \arccos \sqrt{\frac{p+1}{m}}$. Further we suppose, that the following {\bf assumptions} hold.

\medskip

($A_1$) $C_j(\pi, \la_{nk}) \ne 0$, $j = \overline{1, p}$, and $S_j(\pi, \la_{nk}) \ne 0$, $j = \overline{p+1, m}$, for all $n \in \mathbb N$, $k = \overline{1, 4}$.

\smallskip

($A_2$) $C_j(\pi, \mu_{nk}) \ne 0$, $j = \overline{1, p + 1}$, and $S_j(\pi, \mu_{nk}) \ne 0$, $j = \overline{p+1, m}$, for all $n \in \mathbb N$, $k = \overline{1, 2}$.

\medskip

The paper is devoted to the following 2-edge partial inverse problem.

\medskip

{\bf IP.} {\it Given the potentials $\{ \sigma_j \}_{j = \overline{1, m} \backslash \{ 1, p+1 \}}$ and the eigenvalues $\{ \la_{nk} \}_{n \in \mathbb N, \, k = \overline{1, 4}}$, $\{ \mu_{nk} \}_{n \in \mathbb N, \, k = 1, 2}$, find the potentials $\sigma_1$ and $\sigma_{p+1}$.}

\medskip

The paper is organized as follows. {\it Section~2} contains some preliminaries. In {\it Section~3}, we prove the uniqueness theorem for IP. In {\it Section~4}, the constructive procedure for the solution of IP is developed. {\it Appendix A} is devoted to the main technical part of the paper, where we investigate the Riesz-basis property for special systems of functions. In {\it Appendix B}, we provide auxiliary results, concerning entire functions, constructed by their zeros.

Throughout the paper, we use the following notation.
\begin{itemize}
\item $\rho = \sqrt \la$, $\mbox{Re}\, \rho \ge 0$.
\item $B_{2, a}$ is the Paley-Wiener class of entire function of exponential type not greater than $a$, belonging to $L_2(\mathbb R)$.
\item $\mathbb N_0 = \mathbb N \cup \{ 0 \}$.
\item The symbol $C$ stands for different constants, independent on $x$, $\la$, etc.
\end{itemize}

\bigskip

{\large \bf 2. Preliminaries}

\bigskip

The eigenvalues of the problem $L$ coincide with the zeros of {\it the characteristic function}
\begin{equation} \label{defDelta}
	\Delta(\la) = \sum_{j = 1}^p C^{[1]}_j(\pi, \la) \prod_{\substack{i = 1 \\ i \ne j} }^p C_i(\pi, \la) \prod_{k = p+1}^m S_k(\pi, \la) +
	\sum_{j = p+1}^m S_j^{[1]}(\pi, \la) \prod_{i = 1}^p C_i(\pi, \la) \prod_{\substack{k = p+1 \\ k \ne i}}^m S_k(\pi, \la).   	
\end{equation}

Let $L_j$ be the boundary value problem for the Sturm-Liouville equation \eqref{eqv} for each fixed $j = \overline{1, m}$ with the boundary conditions $y_j^{[1]}(0) = 0$, $y_j(\pi) = 0$ for $j = \overline{1, p}$, and
$y_j(0) = y_j(\pi) = 0$ for $j = \overline{p+1, m}$. Denote by $M_j(\la)$ the {\it Weyl functions} of the problems $L_j$:
\begin{equation} \label{defM}
  	M_j(\la) := -\frac{C_j^{[1]}(\pi, \la)}{C_j(\pi, \la)}, \: j = \overline{1, p}, \quad 
  	M_j(\la) := -\frac{S_j^{[1]}(\pi, \la)}{S_j(\pi, \la)}, \: j = \overline{p+1, m}.
\end{equation}
Weyl functions and their generalizations are natural spectral characteristics for different classes of differential operators (see \cite{FIY08, Mar77, FY01}). For each fixed $j  = \overline{1, m}$, the potential $\sigma_j$ can be uniquely recovered from its Weyl function $M_j(\la)$ (see \cite{FIY08}).

Using \eqref{defDelta} and \eqref{defM}, one can easily derive the relation
\begin{equation} \label{sumM}
    \sum_{j = 1}^m M_j(\la) = -\frac{\Delta(\la)}{\prod\limits_{j = 1}^p C_j(\pi, \la) \prod\limits_{j = p+1}^m S_j(\pi, \la)}.
\end{equation}
Taking the assumption ($A_1$) into account, we obtain from \eqref{sumM}:
\begin{equation} \label{defg}
	M_1(\la_{nk}) + M_{p+1}(\la_{nk}) = - \sum_{\substack{j = 2 \\ j \ne p+1}}^m M_j(\la_{nk}) =: g_{nk}, \quad n \in \mathbb N, \: k = \overline{1, 4}.
\end{equation}
It follows from \eqref{defM}, that
\begin{equation} \label{sumMfrac}
M_1(\la) + M_{p+1}(\la) = \frac{D_1(\la)}{D_2(\la)},
\end{equation}
where 
\begin{equation} \label{defD}
D_1(\la) = - (C_1^{[1]}(\pi, \la) S_{p+1}(\pi, \la) + C_1(\pi, \la) S^{[1]}_{p+1}(\pi, \la)),
\quad D_2(\la) =  C_1(\pi, \la) S_{p+1}(\pi, \la).
\end{equation}

\begin{lem} \label{lem:asymptD}
The following relations hold
\begin{equation} \label{asymptD}
D_1(\pi, \la) = -\left( \cos 2 \rho \pi + \int_0^{2 \pi} N(t) \cos \rho t \, dt \right), \quad
D_2(\pi, \la) = \frac{\sin 2 \rho \pi}{2 \rho} + \frac{1}{\rho}\int_0^{2 \pi} K(t) \sin \rho t \, dt,
\end{equation}
where $N$ and $K$ are real-valued functions from $L_2(0, 2 \pi)$.
\end{lem}

\begin{proof}
Using the transformation operators \cite{HM04-transform}, one can obtain the following relations (see \cite{HM03, HM04-2spectra}):
\begin{equation} \label{intCS}
\arraycolsep=1.4pt\def\arraystretch{2.2}
\left.
\begin{array}{ll}
C_1(\pi, \la) & = \cos \rho \pi + \displaystyle\int_0^{\pi} K_1(t) \cos \rho t \, dt,  \\ 
C_1^{[1]}(\pi, \la) & = - \rho \sin \rho \pi + \rho \displaystyle\int_0^{\pi} N_1(t) \sin \rho t \, dt + C_1^{[1]}(\pi, 0), \\ 
S_{p+1}(\pi, \la) & = \dfrac{\sin \rho \pi}{\rho} + \dfrac{1}{\rho} \displaystyle\int_0^{\pi} K_{p+1}(t) \sin \rho t \, dt, \\ 
S_{p+1}^{[1]}(\pi, \la) & = \cos \rho \pi +  \displaystyle\int_0^{\pi} N_{p+1}(t) \cos \rho t \, dt, 
\end{array}
\right\}
\end{equation}
where $K_j, N_j \in L_2(0, \pi)$, $j \in \{1, p+1\}$. Substituting these relations into \eqref{defD}, we get
$$
D_1(\la) = \sin^2 \rho \pi - \cos^2 \rho \pi + F_1(\rho), \quad D_2(\la) = \frac{\cos \rho \pi \sin \rho \pi}{\rho} + \frac{1}{\rho} F_2(\rho), 
$$
where $F_1, F_2 \in B_{2, 2 \pi}$, $F_1$ is even and $F_2$ is odd. Therefore they can be represented in the form
$$
F_1(\rho) = \int_0^{2 \pi} N(t) \cos \rho t \, dt, \quad F_2(\rho) = \int_0^{2 \pi} K(t) \sin \rho t \, dt, \quad N, K \in L_2(0, 2 \pi).
$$
Thus, we arrive at \eqref{asymptD}.
\end{proof}


Now let us study the boundary value problem $L_0$ and its eigenvalues $\{ \mu_{nk} \}_{n \in \mathbb N, \, k = 1, 2}$.
Introduce the Weyl function $M^N_{p+1}(\la) = -\dfrac{C_{p+1}^{[1]}(\pi, \la)}{C_{p+1}(\pi, \la)}$.
Similarly to \eqref{defg}, we obtain the following relation under the assumption ($A_2$):
\begin{equation} \label{defhN}
M_1(\mu_{nk}) + M_{p+1}^N (\mu_{nk}) = - \sum_{\substack{j = 2 \\ j \ne p+1}}^m M_j(\mu_{nk}) =: h_{nk}^N, \quad n \in \mathbb N, \: k = 1, 2.
\end{equation} 
Denote 
\begin{equation} \label{defh}
M_1(\mu_{nk}) + M_{p+1}(\mu_{nk}) =: h_{nk}, \quad n \in \mathbb N, \, k = 1, 2.
\end{equation}
Then
$$
h_{nk}^N - h_{nk} = M_{p+1}^N (\mu_{nk}) - M_{p+1}(\mu_{nk}) = \frac{S_{p+1}^{[1]}(\pi, \mu_{nk}) C_{p+1}(\pi, \mu_{nk}) - C^{[1]}_{p+1}(\pi, \mu_{nk}) S_{p+1}(\pi, \mu_{nk})}{C_{p+1}(\pi, \mu_{nk}) S_{p+1}(\pi, \mu_{nk})}.
$$
Using \eqref{eqv} and \eqref{init}, one can easily show that $S_{p+1}^{[1]}(x, \la) C_{p+1}(x, \la) - C^{[1]}_{p+1}(x, \la) S_{p+1}(x, \la) \equiv 1$ for all $x \in (0, \pi)$, $\la \in \mathbb C$. Thus, we get
\begin{equation} \label{CS}
C_{p+1}(\pi, \mu_{nk}) S_{p+1}(\pi, \mu_{nk}) = \frac{1}{h_{nk}^N - h_{nk}}, \quad n \in \mathbb N, \, k = 1, 2.
\end{equation}

\bigskip

{\large \bf 3. Uniqueness theorem}

\bigskip
Together with $L$ and $L_0$, consider other boundary value problems $\tilde L$ and $\tilde L_0$ of the same form, but with different potentials $\{ \tilde \sigma_j \}_{j = 1}^m$. The values of $m$ and $p$ remains the same. We agree that if a certain symbol $\gamma$ denotes an object related to $L$ or $L_0$, then the corresponding symbol $\tilde \gamma$ with tilde denotes the analogous object related to $\tilde L$ or $\tilde L_0$.

\begin{lem} \label{lem:uniqM}
Let the problems $L$ and $\tilde L$ satisfy the assumption ($A_1$), and let $\sigma_j = \tilde \sigma_j$, $j = \overline{1,m}\backslash \{ 1, p+1 \}$,  $\la_{nk} = \tilde \la_{nk}$, $n \in \mathbb N$, $k = \overline{1, 4}$.
Then $M_1(\la) + M_{p + 1}(\la) \equiv \tilde M_1(\la) + \tilde M_{p+1}(\la)$. 
\end{lem}

\begin{proof}
The relation \eqref{defg} implies
$$
M_1(\la_{nk}) + M_{p+1}(\la_{nk}) = \tilde M_1(\la_{nk}) + \tilde M_{p+1}(\la_{nk}), \quad n \in \mathbb N, \quad k = \overline{1, 4}.
$$
Taking the relation \eqref{sumMfrac} into account, we get
$$
D_1(\la_{nk}) \tilde D_2(\la_{nk}) - \tilde D_1(\la_{nk}) D_2(\la_{nk}) = 0.
$$
Thus, the function
$$
H(\la) = D_1(\la) \tilde D_2(\la) - \tilde D_1(\la) D_2(\la) 
$$
has zeros at the points $\{ \la_{nk} \}_{n \in \mathbb N, \, k = \overline{1, 4}}$. 
Construct the entire function
$$
P(\la) := \prod_{k = 1}^4 \prod_{n = 1}^{\iy} \left(1 - \frac{\la}{\la_{nk}} \right).
$$
(The case $\la_{nk} = 0$ requires minor modifications). Obviously, $\dfrac{H(\la)}{P(\la)}$ is an entire function of order $\frac{1}{2}$. According to Lemma~\ref{lem:asymptD} and Corollary~\ref{cor:prodla} from Appendix~B, the estimate $\left|\dfrac{H(\la)}{P(\la)} \right| \le C$ holds for $\la = \rho^2$, $\eps < \arg \rho < \pi - \eps$. Applying Phragmen-Lindel\"of's theorem \cite{BFY} and Liouville's theorem, we conclude that $H(\la) \equiv C P(\la)$. 
By virtue of Lemma~\ref{lem:asymptD}, the function $\rho H(\rho^2)$ belongs to the Paley-Wiener class $B_{2, 4 \pi}$, as a function of $\rho$, but $\rho P(\rho^2) \not \in B_{2, 4 \pi}$ (see \eqref{reprP}). Consequently, $C = 0$ and $H(\la) \equiv 0$. Thus, $\dfrac{D_1(\la)}{D_2(\la)} \equiv \dfrac{\tilde D_1(\la)}{\tilde D_2(\la)}$, and the lemma is proved.
\end{proof}

\begin{remark}
If together with the potentials $\{ \sigma_j \}_{j = \overline{1, m} \backslash \{ 1, p + 1\}}$ even more eigenvalues of $L$ are given, than the collection $\{ \la_{nk} \}_{n \in \mathbb N, \, k = \overline{1, 4}}$ contains, we can not obtain more information, than the sum of the Weyl functions $M_1(\la) + M_{p+1}(\la)$. We need some additional data, to ``separate'' the potentials $\sigma_1$ and $\sigma_{p+1}$.
\end{remark}

\begin{thm} \label{thm:uniq}
Let $\sigma_j = \tilde \sigma_j$, $j = \overline{1,m}\backslash \{ 1, p+1 \}$,
$\la_{nk} = \tilde \la_{nk}$ for $n \in \mathbb N$, $k = \overline{1, 4}$, and $\mu_{nk} = \tilde \mu_{nk}$ for $n \in \mathbb N$, $k = 1, 2$.
Assume that ($A_1$), ($A_2$) hold for the problems $L$, $L_1$, $\tilde L$, $\tilde L_1$.
Then $\sigma_1 = \tilde \sigma_1$ and $\sigma_{p+1} = \tilde \sigma_{p+1}$ in $L_2(0, \pi)$. Thus, the solution of IP is unique.
\end{thm}

\begin{proof}
By virtue of Lemma~\ref{lem:uniqM} 
\begin{equation} \label{sumMeq}
M_1(\la) + M_{p+1}(\la) = \tilde M_1(\la) + \tilde M_{p+1}(\la).
\end{equation} 
The relations \eqref{defhN}, \eqref{defh} and similar relations for $\tilde L$ and $\tilde L_0$ imply
$h_{nk} = \tilde h_{nk}$, $h_{nk}^N = \tilde h_{nk}^N$ for $n \in \mathbb N$, $k = 1, 2$. Taking \eqref{CS} into account, we conclude that the entire function function
$$
H(\la) := C_{p + 1}(\pi, \la) S_{p+1}(\pi, \la) - \tilde C_{p+1}(\pi, \la) \tilde S_{p+1}(\pi, \la)
$$
has zeros $\{ \mu_{nk} \}_{n \in \mathbb N, \, k = 1, 2}$.
Similarly to \eqref{intCS}, one can derive the relations 
\begin{equation} \label{intCS2}
\arraycolsep=1.4pt\def\arraystretch{2.2}
\left.
\begin{array}{ll}
C_{p + 1}(\pi, \la) & = \cos \rho \pi + \displaystyle\int_0^{\pi} T_{p+1}(t) \cos \rho t \, dt, \\
S_{p + 1}(\pi, \la) & = \dfrac{\sin \rho \pi}{\rho} + \dfrac{1}{\rho} \displaystyle\int_0^{\pi} K_{p+1}(t) \sin \rho t \, dt,
\end{array}
\right\}
\end{equation}
where $T_{p+1}, K_{p+1} \in L_2(0, \pi)$. Hence $|H(\la)| \le C |\rho|^{-1} \exp(2 |\mbox{Im}\, \rho| \pi)$ for $|\rho| \ge \rho^* > 0$.

Construct the function
$$
P(\la) := \prod_{k = 1}^2 \prod_{n = 1}^{\iy} \left( 1 - \frac{\la}{\mu_{nk}}\right).
$$
(The case $\mu_{nk} = 0$ requires minor changes). In view of the asymptotics \eqref{asymptmu}, one can apply Corollary~\ref{cor:shift} from Appendix~B to $P(\la)$. Consequently, the entire function $\dfrac{H(\la)}{P(\la)}$ admits the estimate $\left| \dfrac{H(\la)}{P(\la)}\right| \le \dfrac{C}{|\rho|}$ for $\la = \rho^2$, $\eps < \arg \rho < \pi - \eps$, $|\rho| > \rho^*$ for some $\eps > 0$ and $\rho^* > 0$. By Phragmen-Lindel\"of's and Liouville's theorems, we get $H(\la) \equiv 0$. Hence $C_{p+1}(\pi, \la) S_{p+1}(\pi, \la) \equiv \tilde C_{p+1}(\pi, \la) \tilde S_{p+1} (\pi, \la)$. 

The functions $C_{p + 1}(\pi, \la)$ and $S_{p + 1}(\pi, \la)$ have real zeros $\{ \nu_n \}_{n \in \mathbb N_0}$ and $\{ \theta_n \}_{n \in \mathbb N}$, which interlace~\cite{HM04-2spectra}:
\begin{equation} \label{interlace}
	\nu_0 < \theta_1 < \nu_1 < \theta_2 < \nu_2 < \dots
\end{equation}
The same assertion is valid for $\tilde C_{p + 1}(\pi, \la)$ and $\tilde S_{p + 1}(\pi, \la)$. Consequently, $\nu_n = \tilde \nu_n$, for all $n \in \mathbb N_0$ and $\theta_n = \tilde \theta_n$ for $n \in \mathbb N$. It has been proved in \cite{HM04-2spectra}, that the two spectra $\{ \nu_n \}_{n \in \mathbb N_0}$ and $\{ \theta_n \}_{n \in \mathbb N}$ determine the potential $\sigma_{p + 1}$ uniquely. Hence $M_{p + 1}(\la) \equiv \tilde M_{p + 1}(\la)$. Together with \eqref{sumMeq}, this yields $M_1(\la) \equiv \tilde M_1(\la)$. The Weyl function $M_1(\la)$ determines the potential $\sigma_1$ uniquely (see \cite{FIY08}). Thus, $\sigma_1 = \tilde \sigma_1$ and $\sigma_{p + 1} = \tilde \sigma_{p + 1}$ in $L_2(0, \pi)$.
\end{proof}

\bigskip

{\large \bf 4. Solution of IP}

\bigskip

In this section, we develop a constructive algorithm for the solution of IP. First, we show that, using the eigenvalues $\{ \la_{nk} \}_{n \in \mathbb N, \, k = \overline{1, 4}}$, one can obtain a sequence of coefficients of some vector function $f(t)$ with respect to the specially constructed Riesz basis. Recovering $f(t)$ from its coefficients, we can find the sum $M_1(\la) + M_{p + 1}(\la)$. Then we add the part of the second spectrum $\{ \mu_{nk} \}_{n \in \mathbb N, \, k = 1, 2}$ and find the potentials $\sigma_1$ and $\sigma_{p+1}$.

Substituting \eqref{sumMfrac} and \eqref{asymptD} into \eqref{defg}, we obtain
\begin{equation} \label{NK}
\frac{\rho_{nk}}{g_{nk}} \int_0^{2 \pi} N(t) \cos \rho_{nk} t \, dt +  \int_0^{2 \pi} K(t) \sin \rho_{nk} t \, dt = f_{nk}, \quad n \in \mathbb N, \: k = \overline{1, 4},  
\end{equation}
\begin{equation} \label{deff}
f_{nk} := \frac{\rho_{nk}}{g_{nk}}\cos 2 \rho_{nk} \pi + \frac{1}{2} \sin \rho_{nk} \pi.
\end{equation}

For simplicity, we assume that

\medskip

($A_3$) the numbers $\{ \la_{nk} \}_{n \in \mathbb N, \, k = \overline{1, 4}}$ are distinct and positive;

\smallskip

($A_4$) $g_{nk} \ne 0$, $n \in \mathbb N$, $k = \overline{1, 4}$.

\medskip

These assumptions are not principal. The case of multiple eigenvalues was discussed in \cite{Bond17}, while the other conditions can be easily achieved by a shift $q_j \to q_j + C$, $j = \overline{1, m}$.

Denote
\begin{equation} \label{defv}
f(t) = \begin{bmatrix} N(t) \\ K(t) \end{bmatrix}, \quad
v_{nk}(t) = \begin{bmatrix} \frac{\rho_{nk}}{g_{nk}}\cos \rho_{nk} t \\ \sin \rho_{nk} t \end{bmatrix}, \quad n \in \mathbb N, \: k = \overline{1, 4}.
\end{equation}

Consider the real Hilbert space $\mathcal{H} := L_2(0, 2\pi) \oplus L_2(0, 2\pi)$.
The scalar product and the norm in $\mathcal H$ are defined as follows
$$
   (g, h)_{\mathcal H} = \int_0^{2 \pi} ( g_1(t) h_1(t) + g_2(t) h_2(t)) \, dt, \quad
   \| g \|_{\mathcal H} = \sqrt{\int_0^{2 \pi} (g_1^2(t) + g_2^2(t)) \, dt}, 
$$
$$
   g = \begin{bmatrix} g_1 \\ g_2 \end{bmatrix}, \quad h = \begin{bmatrix} h_1 \\ h_2 \end{bmatrix}, \quad g, h \in \mathcal H.    
$$
One can rewrite the relation \eqref{NK} in the form
\begin{equation} \label{scal}
(f, v_{nk})_{\mathcal H} = f_{nk}, \quad n \in \mathbb N, \quad k = \overline{1, 4}.
\end{equation}

In Appendix A, we will prove the following theorem.

\begin{thm} \label{thm:Riesz}
Under the assumptions ($A_1$), ($A_3$), ($A_4$),
the system $\{ v_{nk} \}_{n \in \mathbb N, \, k = \overline{1, 4}}$ is a Riesz basis in $\mathcal H$.
\end{thm}

In view of Theorem~\ref{thm:Riesz} and the relation \eqref{scal}, the numbers $f_{nk}$ are the coordinates of the vector function $f$ with respect to the Riesz basis, biorthonormal to $\{ v_{nk} \}_{n \in \mathbb N, \, k = \overline{1, 4}}$. Given $\{ v_{nk} \}_{n \in \mathbb N, \, k = 1, 2}$ and $\{ f_{nk} \}_{n \in \mathbb N, \, k = 1, 2}$, we can recover $f$ uniquely. Consequently, we know $N(t)$ and $K(t)$, and can find the sum $M_1(\la) + M_{p + 1}(\la)$ via \eqref{asymptD} and \eqref{sumMfrac}.

Now given $\{ \mu_{nk} \}_{n \in \mathbb N, \, k = 1, 2}$, one can find $h_{nk}^N$ and $h_{nk}$ via \eqref{defhN} and \eqref{defh}, respectively. Consider the relation \eqref{CS}. It follows from \eqref{intCS2}, that
\begin{equation} \label{prodCS}
C_{p + 1}(\pi, \la) S_{p + 1}(\pi, \la) = \frac{\sin 2 \rho \pi}{2 \rho} + \frac{1}{\rho} \int_0^{2 \pi} T(t) \sin \rho t \, dt, 
\end{equation}
where $T \in L_2(0, 2 \pi)$. Substituting \eqref{prodCS} into \eqref{CS}, we obtain
$$
C_{p + 1}(\pi, \mu_{nk}) S_{p + 1}(\pi, \mu_{nk}) = \frac{\sin 2 \sqrt{\mu_{nk}}\pi}{2\sqrt{\mu_{nk}}} + \frac{1}{\sqrt{\mu_{nk}}} \int_0^{2\pi} T(t) \sin \sqrt{\mu_{nk}}t \, dt = 
\frac{1}{h_{nk}^N - h_{nk}} 
$$
for $n \in \mathbb N$, $k = 1, 2$. Then we derive the following system of equations
\begin{equation} \label{sysT}
\int_0^{2 \pi} T(t) \sin \sqrt{\mu_{nk}} t \, dt = \frac{\sqrt{\mu_{nk}}}{h_{nk}^N - h_{nk}} - \frac{1}{2} \sin 2 \sqrt{\mu_{nk}} \pi, \quad n \in \mathbb N, \: k = 1, 2.
\end{equation}
Impose an additional assumption:

\smallskip

($A_5$) the numbers $\{ \mu_{nk} \}_{n \in \mathbb N,\, k = 1, 2}$ are distinct and positive.

\smallskip

The following theorem will be proved in Appendix A.

\begin{thm} \label{thm:Riesz2}
Under the assumptions ($A_2$), ($A_5$), the system $\{ \sin \sqrt{\mu_{nk}} t \}_{n \in \mathbb N, \, k = 1, 2}$ is a Riesz basis in $L_2(0, 2 \pi)$.
\end{thm}

Thus, one can solve the system \eqref{sysT} uniquely, recovering the function $T$ from its coefficients with respect to the Riesz basis. Then using \eqref{prodCS}, one can find the product $C_{p+1}(\pi, \la) S_{p+1}(\pi, \la)$ and the interlacing zeros $\{ \nu_n \}_{n \in \mathbb N_0}$ and $\{ \theta_n \}_{n \in \mathbb N}$ of the entire functions $C_{p+1}(\pi, \la)$ and $S_{p+1}(\pi, \la)$, respectively. These data can be used for reconstruction of the potential $\sigma_{p+1}$. 

Summarizing the results of this section, we arrive at the following algorithm for the solution of IP.

\medskip

{\bf Algorithm.} Let the potentials $\{ \sigma_j \}_{j = \overline{1, m} \backslash \{ 1, p + 1 \}}$ and the eigenvalues $\{ \la_{nk} \}_{n \in \mathbb N\, k = \overline{1, 4}}$, $\{ \mu_{nk} \}_{n \in \mathbb N,\, k = 1, 2}$ be given, and let the assumptions ($A_1$)-($A_5$) be satisfied.

\begin{enumerate}
\item Using the potentials $\{ \sigma_j \}_{j = \overline{1, m} \backslash \{ 1, p + 1 \}}$, construct the Weyl functions $M_j(\la)$ for $j = \overline{1, m} \backslash \{ 1, p + 1 \}$ via \eqref{defM}.
\item Construct the numbers $g_{nk}$, $f_{nk}$ and the vector functions $v_{nk}$ for $n \in \mathbb N$, $k = \overline{1, 4}$ via \eqref{defg}, \eqref{deff} and \eqref{defv}.
\item According to \eqref{scal}, recover the vector function $f(t) = \begin{bmatrix} N(t) \\ K(t) \end{bmatrix}$ from its coordinates $f_{nk}$ with respect to the Riesz basis.
\item Using $N(t)$ and $K(t)$, recover the sum $M_1(\la) + M_{p+1}(\la)$ via \eqref{sumMfrac}, \eqref{asymptD}.
\item Find the numbers $h_{nk}^N$ and $h_{nk}$ for $n \in \mathbb N$, $k = 1, 2$, using \eqref{defhN} and \eqref{defh}.
\item Construct the system \eqref{sysT} and find the function $T(t)$ from this system, recovering it from the coordinates with respect to the Riesz basis.
\item Using $T(t)$, construct the product $C_{p+1}(\pi, \la) S_{p+1}(\pi, \la)$ via \eqref{prodCS}.
\item Find the zeros of the product $C_{p+1}(\pi, \la) S_{p+1}(\pi, \la)$ and divide it into two sequences $\{ \nu_n \}_{n \in \mathbb N_0}$ and $\{ \theta_n \}_{n \in \mathbb N}$, interlacing according to \eqref{interlace}.
\item Construct the potential $\sigma_{p+1}$ by two spectra $\{ \nu_n \}_{n \in \mathbb N_0}$ and $\{ \theta_n \}_{n \in \mathbb N}$ (see \cite{HM04-2spectra}).
\item Find $M_{p+1}(\la)$ by $\sigma_{p+1}(\la)$, and then $M_1(\la)$ from the sum $M_1(\la) + M_{p+1}(\la)$.
\item Construct the potential $\sigma_1$ by the Weyl function $M_1(\la)$ (see \cite{FIY08, FY01}).  
\end{enumerate}

\bigskip

{\large \bf Appendix A. Riesz bases}

\bigskip

The goal of this section is to prove the important Theorems~\ref{thm:Riesz} and~\ref{thm:Riesz2}.
We start with the analysis of the auxiliary systems $\mathcal S := \{ \sin (n + \be) t \}_{n \in \mathbb Z}$ and $\mathcal C := \{ \cos (n + \be) t \}_{n \in \mathbb Z}$ in $L_2(0, 2 \pi)$. Here $\be$ is an arbitrary number from $(0, \frac{1}{2})$.

\begin{lem}
The systems $\mathcal S$ and $\mathcal C$ are complete in $L_2(0, 2\pi)$.
\end{lem}

\begin{proof}
Let us prove the assertion of the lemma for the system $\mathcal S$. The proof for $\mathcal C$ is similar.

Suppose that, on the contrary, the system $\mathcal S$ is not complete. Then there exists a nonzero function $h$ from $L_2(0, 2 \pi)$, such that
$$
\int_0^{2 \pi} h(t) \sin (n + \be) t \, dt = 0, \quad n \in \mathbb Z.
$$
Hence the odd entire function
$$
H(\rho) := \int_0^{2 \pi} h(t) \sin \rho t \, dt
$$
has zeros $\{ \pm (n + \be) \}_{n \in \mathbb Z} \cup \{ 0 \}$, which coincide with the zeros of the function $D(\rho) := \rho (\cos^2 \rho \pi - \cos^2 \be \pi)$. Clearly, the function $\dfrac{H(\rho)}{D(\rho)}$ is entire and $\dfrac{H(\rho)}{D(\rho)} = O(\rho^{-1})$, as $|\rho| \to \iy$. By virtue of Liouville's theorem, $H(\rho) \equiv 0$ and $h = 0$. The contradiction proves the lemma.
\end{proof}

\begin{lem}
For an arbitrary sequence $\{ c_n \}_{n \in \mathbb Z}$ from $l_2$, the following estimates hold
\begin{equation} \label{RBestS}
\pi (1 - \cos 2 \be \pi) \sum_{n = -\iy}^{\iy} c_n^2 \le \left\| \sum_{n = -\iy}^{\iy} c_n \sin (n + \be) t \right\|_2^2 \le \pi (1 + \cos 2 \be \pi) \sum_{n = -\iy}^{\iy} c_n^2,
\end{equation}
\begin{equation} \label{RBestC}
\pi (1 - \cos 2 \be \pi) \sum_{n = -\iy}^{\iy} c_n^2 \le \left\| \sum_{n = -\iy}^{\iy} c_n \cos (n + \be) t \right\|_2^2 \le \pi (1 + \cos 2 \be \pi) \sum_{n = -\iy}^{\iy} c_n^2.
\end{equation}

Thus, the systems $\mathcal S$ and $\mathcal C$ are Riesz bases in $L_2(0, 2 \pi)$.
\end{lem}

\begin{proof}
Without loss of generality, consider real sequences $\{ c_n \}_{n \in \mathbb Z}$. Similarly to the proof of \cite[Theorem 3.1]{Sedl03}, we derive
\begin{multline*}
\left\| \sum_{n = -\iy}^{\iy} c_n \sin (n + \be) t \right\|_2^2 = \int_0^{2 \pi} \left( \sum_{n = -\iy}^{\iy} \sum_{k = -\iy}^{\iy} c_n c_k \sin (n + \be) t \sin (k + \be) t \right)  \\
= \frac{1}{2} \int_0^{2 \pi} \left( \sum_{n = -\iy}^{\iy} \sum_{k = -\iy}^{\iy} c_n c_k (\cos ( n - k) t - \cos(n + k + 2 \be) t) \right) \, dt \\
= \pi \sum_{n = -\iy}^{\iy} c_n^2 - \frac{1}{2} \sin 4 \be \pi \sum_{n = -\iy}^{\iy} \sum_{k = -\iy}^{\iy} \frac{c_n c_k}{n + k + 2 \be}.
\end{multline*}
Consider the bilinear form
$$
A = \sum_{n = -\iy}^{\iy} \sum_{k = -\iy}^{\iy} a_{nk} c_n c_k, \quad a_{nk} = \frac{1}{n + k + 2 \be}.
$$
Let us calculate its norm (see \cite{HLP}):
$$
B = \sum_{i = -\iy}^{\iy} \sum_{j = -\iy}^{\iy} b_{ij} x_i x_j,
$$
\vspace*{-8mm}
\begin{multline*}
b_{ij} = \sum_{k = -\iy}^{\iy} a_{ik} a_{jk} = \sum_{k = -\iy}^{\iy} \frac{1}{(i + k + 2 \be) (j + k + 2 \be)} \\ = \frac{1}{j - i} \sum_{k = -\iy}^{\iy} \left( \frac{1}{i + k + 2 \be} - \frac{1}{j + k + 2 \be}\right) = 0, \quad i \ne j, 
\end{multline*}
$$
b_{ii} = \sum_{k = -\iy}^{\iy} a_{ik}^2 = \sum_{n = -\iy}^{\iy} \frac{1}{(n + 2 \be)^2} = \frac{\pi^2}{\sin^2 2 \be \pi}.
$$
Consequently,
$$
B = \frac{\pi^2}{\sin^2 2 \be \pi} \sum_{i = -\iy}^{\iy} x_i^2,
\qquad 
|A| \le \frac{\pi}{\sin 2 \be \pi} \sum_{n = -\iy}^{\iy} c_n^2.
$$
Finally, we obtain
$$
\pi (1 - \cos 2 \be \pi) \sum_{n = -\iy}^{\iy} c_n^2 \le \pi \sum_{n = -\iy}^{\iy} c_n^2  - \frac{1}{2} \sin 4 \pi \be \cdot A \le \pi (1 + \cos 2 \be \pi) \sum_{n = -\iy}^{\iy} c_n^2,
$$
so we arrive at \eqref{RBestS}. The estimate \eqref{RBestC} can be proved similarly.
\end{proof}

\begin{proof}[Proof of Theorem~\ref{thm:Riesz2}]
Since the eigenvalues $\{ \mu_{nk} \}_{n \in \mathbb N, \, k = 1, 2}$ satisfy the asymptotic formulas \eqref{asymptmu}, we have
$$
\sin \sqrt{\mu_{n1}} t = \sin \left(n - 1 + \frac{\al_1}{\pi}\right) t + \varkappa_n, \quad
\sin \sqrt{\mu_{n2}} t = \sin \left(n - \frac{\al_1}{\pi}\right) t + \varkappa_n.
$$
Thus, the considered system $\{ \sin \sqrt{\mu_{nk}} \}_{n \in \mathbb N, \, k = 1, 2}$ is $l_2$-close to the Riesz basis $\mathcal S$ for $\be = \frac{\al_1}{\pi} \in \left(0, \frac{\pi}{2}\right)$. In order to prove the Riesz-bacisity of $\{ \sin \sqrt{\mu_{nk}} \}_{n \in \mathbb N, \, k = 1, 2}$, it remains to show that this system is complete in $L_2(0, 2 \pi)$.

Suppose that the contrary holds, i.e. there exists a function $h \ne 0$ from $L_2(0, 2 \pi)$, such that 
$$
\int_0^{2 \pi} h(t) \sin \sqrt{\mu_{nk}} t \, dt = 0, \quad n \in \mathbb N, \: k = 1, 2. 
$$
Consequently, the entire function
$$
H(\la) := \frac{1}{\rho} \int_0^{2 \pi} h(t) \sin \rho t \, dt
$$
has zeros $\{ \mu_{nk} \}_{n \in \mathbb N, \, k = 1, 2}$. Obviously, the estimate $|H(\la)| \le C |\rho|^{-1} \exp(2 |\mbox{Im}\, \rho| \pi)$ holds for $|\rho| \ge \rho^* > 0$. Further one can repeat the arguments from the proof of Theorem~\ref{thm:uniq}, and show that $H(\la) \equiv 0$. Thus, $h = 0$,
and we arrive at the contradiction. Hence the system $\{ \sin \sqrt{\mu_{nk}} \}_{n \in \mathbb N, \, k = 1, 2}$ is complete in $L_2(0, 2 \pi)$.
\end{proof}

Now we proceed to the system $\{ v_{nk} \}_{n \in \mathbb N, \, k = \overline{1, 4}}$.
Denote
$$
v_{n1}^0(t) = \begin{bmatrix} -\frac{1}{2} \tan 2 \al \cos (n - 1 + \frac{\al}{\pi}) t \\ \sin (n - 1 + \frac{\al}{\pi} ) t \end{bmatrix}, \quad
v_{n2}^0(t) = \begin{bmatrix} \frac{1}{2} \tan 2 \al \cos (n - \frac{\al}{\pi}) t \\ \sin (n - \frac{\al}{\pi}) t \end{bmatrix},
$$
$$
v_{n3}^0(t) = \begin{bmatrix} 0 \\ \sin (n - \frac{1}{2} ) t \end{bmatrix}, \quad
v_{n4}^0(t) = \begin{bmatrix} 0 \\ \sin n t \end{bmatrix}, \quad n \in \mathbb N.
$$

\begin{lem} \label{lem:estv}
The sequence $\{ v_{nk} \}$ is $l_2$-close to the sequence $\{ v_{nk}^0 \}$ in $\mathcal H$, i.e.
$$
\{ \| v_{nk} - v_{nk}^0 \| \}_{n \in \mathbb N, \, k = \overline{1, 4}} \in l_2.
$$
\end{lem}

\begin{proof}
Using the relations \eqref{asymptrho}, \eqref{defg}, \eqref{sumMfrac} and \eqref{asymptD}, we obtain
\begin{gather*}
g_{n1}^{-1} = -\frac{1}{2 n} \tan 2 \al + \frac{\varkappa_n}{n}, \quad g_{n2}^{-1} = \frac{1}{2 n} \tan 2 \al + \frac{\varkappa_n}{n}, \\
g_{n3}^{-1} = \frac{\varkappa_n}{n}, \quad g_{n4}^{-1} = \frac{\varkappa_n}{n}, \quad n \in \mathbb N.
\end{gather*}
Substituting these estimates together with \eqref{asymptrho} into \eqref{defv}, we arrive at the assertion of the lemma.
\end{proof}

\begin{lem} \label{lem:Rieszv0}
The system $\{ v_{nk}^0 \}_{n \in \mathbb N, \, k = \overline{1, 4}}$ is a Riesz basis in $\mathcal H$.
\end{lem}

\begin{proof}
Let us construct a linear bounded operator $A \colon \mathcal H \to \mathcal H$ with a bounded inverse, such that the system $\{ A v_{nk}^0 \}_{n \in \mathbb N, \, k = \overline{1, 4}}$ is a Riesz basis. Put
$$
A v = A \begin{bmatrix} v_1 \\ v_2 \end{bmatrix} = 
\begin{bmatrix} v_1 \\ v_2 + g \end{bmatrix}, 
\quad 
A^{-1} v = \begin{bmatrix} v_1 \\ v_2 - g \end{bmatrix}, 
$$
where 
$$
g(t) := 2 \cot 2 \al \sum_{n = -\iy}^{\iy} c_n \sin (n + \be), \quad \be := \frac{\al}{\pi},
$$
and $\{ c_n \}_{n \in \mathbb Z}$ are the coordinates of $v_1$ with respect to the Riesz basis $\mathcal C$:
$$
v_1(t) = \sum_{n = -\iy}^{\iy} c_n \cos (n + \be) t.
$$
Using the estimates \eqref{RBestS} and \eqref{RBestC}, one can easily show that the operators $A$ and $A^{-1}$ are bounded in $\mathcal H$. 
Furthermore, we have
\begin{gather*}
(A v_{n1}^0)(t) = -\frac{1}{2} \tan 2 \al \begin{bmatrix} \cos (n - 1 + \be) t \\ 0\end{bmatrix}, \quad
(A v_{n2}^0)(t) = \frac{1}{2} \tan 2 \al \begin{bmatrix} \cos (- n + \be) t \\ 0 \end{bmatrix}, \\
(A v_{n3}^0)(t) = \begin{bmatrix} 0 \\ \sin (n - \frac{1}{2}) t \end{bmatrix}, \quad
(A v_{n4}^0)(t) = \begin{bmatrix} 0 \\ \sin n t \end{bmatrix}.
\end{gather*}
Since the systems $\mathcal C$ and $\{ \sin n t \}_{n \in \mathbb N} \cup \{ \sin (n - \frac{1}{2}) t \}_{n \in \mathbb N}$ are Riesz bases in $L_2(0, 2 \pi)$, the system $\{ A v_{nk}^0 \}_{n \in \mathbb N, \, k = \overline{1, 4}}$ is a Riesz basis in $\mathcal H$.
\end{proof}

\begin{lem} \label{lem:complete}
Under the assumptions ($A_1$), ($A_3$) and ($A_4$), the system $\{ v_{nk} \}_{n \in \mathbb N, \, k = \overline{1, 4}}$ is complete in $\mathcal H$.
\end{lem}

\begin{proof}
Suppose that functions $w_1, w_2 \in L_2(0, 2 \pi)$ are such that
\begin{equation} \label{smeqw}
\int_0^{2 \pi} \left(w_1(t) \frac{\rho_{nk}}{g_{nk}} \cos \rho_{nk} t + w_2(t) \sin \rho_{nk} t \right) \, dt = 0, \quad
n \in \mathbb N, \quad k = \overline{1, 4}.
\end{equation}
Recall that $g_{nk} = M_1(\la_{nk}) + M_{p + 1}(\la_{nk})$ and $M_1(\la) + M_{p + 1}(\la) = \dfrac{D_1(\la)}{D_2(\la)}$. In view of the assumptions ($A_3$) and ($A_4$), $\rho_{nk} \ne 0$ and $g_{nk} \ne 0$ for $n \in \mathbb N$, $k = \overline{1, 4}$. The assumption ($A_1$) together with \eqref{defD} implies $D_2(\la_{nk}) \ne 0$.
Consequently, the entire function
\begin{equation} \label{defW}
W(\la) := \int_0^{2 \pi} \left( w_1(t) D_2(\la) \cos \rho t + w_2(t) D_1(\la) \frac{\sin \rho t}{\rho} \right) \, dt
\end{equation}
has zeros at the points $\{ \la_{nk} \}_{n \in \mathbb N, \, k = \overline{1, 4}}$. It follows from \eqref{asymptD} and \eqref{defW}, that 
\begin{equation} \label{estW}
W(\la) = O(|\rho|^{-1}\exp(4 |\mbox{Im}\,\rho|\pi)), \quad |\rho| \ge \rho^* > 0.
\end{equation}

Construct the function
\begin{equation*} 
P(\la) := \prod_{k = 1}^4 \prod_{n = 1}^{\iy} \left( 1 - \frac{\la}{\la_{nk}} \right).
\end{equation*}
Note that $\la_{nk} \ne 0$ due to ($A_3$). Clearly, the function $\dfrac{W(\la)}{P(\la)}$ is entire.
The estimate \eqref{estW} and Corollary~\ref{cor:prodla} from Appendix~B yield
$\dfrac{W(\la)}{P(\la)} = O(1)$ for $\la = \rho^2$, $\eps < \arg \rho < \pi - \eps$. Applying Phragmen-Lindel\"of's and Liouville's theorems \cite{BFY}, we conclude that $W(\la) \equiv C P(\la)$. Using \eqref{defW}, one can easily show that $\rho W(\rho^2) \in B_{2, 4\pi}$ (as a function of $\rho$). However, the expression \eqref{reprP} implies $\rho P(\rho^2) \not \in B_{2, 4 \pi}$. Hence $C = 0$ and $W(\la) \equiv 0$.

Let $\{ \tau_n \}_{n \in \mathbb N}$ be the sequence of zeros of the function $D_2(\la)$. Using \eqref{defD}, one can easily check, that $D_1(\tau_n) \ne 0$. Consequently, the function
\begin{equation} \label{defH}
H(\la) := \int_0^{2 \pi} w_2(t) \frac{\sin \rho t}{\rho} \, dt
\end{equation}
has zeros at the points $\{ \tau_n \}_{n \in \mathbb Z}$ (if $\tau_n$ is a multiple zero of $D_2(\la)$, then it is also a multiple zero of $H(\la)$ with the same multiplicity). 
Thus, the function $\dfrac{H(\la)}{D_2(\la)}$ is entire. 

It follows from \eqref{defH} and \eqref{asymptD}, that $\dfrac{H(\la)}{D_2(\la)} = O(1)$ as $|\la| \to \iy$. By Liouville's theorem, $H(\la) \equiv C D_2(\la)$. However, $\rho H(\rho^2) \in B_{2, 2\pi}$ and $\rho D_2(\rho^2) \not \in B_{2, 2 \pi}$. Therefore $H(\la) \equiv 0$ and, consequently, $w_2 = 0$ in $L_2(0, \pi)$. 
Using \eqref{defW} and the relation $W(\la) \equiv 0$, we conclude that $w_1 = 0$. In view of \eqref{smeqw}, the system $\{ v_{nk} \}_{n \in \mathbb N, \, k = \overline{1, 4}}$ is complete in $\mathcal H$.
\end{proof}

\begin{proof}[Proof of Theorem~\ref{thm:Riesz}]
The assertion of the theorem immediately follows from Lemmas~\ref{lem:estv},~\ref{lem:Rieszv0} and~\ref{lem:complete}. 
\end{proof}

\bigskip

{\large \bf Appendix B. Entire functions}

\bigskip

Here we discuss entire functions, constructed as infinite products by their zeros with a certain asymptotic behavior. We derive some relations, which can be used for evaluation of these functions. Our analysis is based on the following result.

\begin{lem}[\cite{BB16}] \label{lem:prod}
Let 
$$
\rho_n = n + \varkappa_n, \quad n \in \mathbb Z, 
$$
be arbitrary complex numbers, and 
$$
P(\rho) := \pi (\rho - \rho_0) \prod_{\substack{n = -\iy \\n \ne 0}}^{\iy} \frac{\rho_n - \rho}{n} \exp\left(\frac{\rho}{n}\right).
$$
Then $P(\rho)$ can be represented in the form
$$
P(\rho) = \sin \rho \pi + \int_0^{\pi} ( w_1(t) \sin \rho t + w_2(t) \cos \rho t) \, dt, 
$$
where $w_1, w_2 \in L_2(0, \pi)$.
\end{lem}

Lemma~\ref{lem:prod} has the following corollaries. We prove only Corollary~\ref{cor:shift}, since the proofs of Corollaries~\ref{cor:sin} and~\ref{cor:cos} exploit similar ideas.

\begin{cor} \label{cor:sin}
The function
$$
P(\la) := \prod_{n = 1}^{\iy} \left( 1 - \frac{\la}{\la_n}\right),
$$
where 
$$
\la_n = \rho_n^2 \ne 0, \quad \rho_n = n + \varkappa_n, \quad n \in \mathbb N,
$$
admits the representation
$$
P(\la) = \frac{C \sin \rho \pi}{\rho} + \frac{1}{\rho} \int_0^{\pi} w(t) \sin \rho t \, dt, \quad w \in L_2(0, \pi).
$$
\end{cor}

\begin{cor} \label{cor:cos}
The function
$$
P(\la) := \prod_{n = 1}^{\iy} \left( 1 - \frac{\la}{\la_n}\right),
$$
where 
$$
\la_n = \rho_n^2 \ne 0, \quad \rho_n = n - \frac{1}{2} + \varkappa_n, \quad n \in \mathbb N,
$$
admits the representation
$$
P(\la) = C \cos \rho \pi + \int_0^{\pi} w(t) \cos \rho t \, dt, \quad w \in L_2(0, \pi).
$$
\end{cor}

\begin{cor} \label{cor:shift}
The function
\begin{equation} \label{prodP}
P(\la) := \prod_{n = 0}^{\iy} \left(1 - \frac{\la}{\la_n^+} \right) \prod_{n = 1}^{\iy} \left(1 - \frac{\la}{\la_n^-} \right),
\end{equation}
where
\begin{align*}
\la_n^+ & = (\rho_n^+)^2 \ne 0, \quad \rho_n^+ = n + a + \varkappa_n, \quad n \in \mathbb N_0, \\
\la_n^- & = (\rho_n^-)^2 \ne 0, \quad \rho_n^- = n - a + \varkappa_n, \quad n \in \mathbb N, \\
\end{align*}
admits the representation
$$
P(\la) = C (\cos 2 \rho \pi - \cos 2 a \pi) + \int_0^{2 \pi} w(t) \cos \rho t\, dt, \quad w \in L_2(0, 2 \pi).
$$
\end{cor}

\begin{proof}
Denote $\rho_{-n}^+ = -\rho_n^-$ for $n \in \mathbb N$ and $\rho_{-n}^- = - \rho_n^+$ for $n \in \mathbb N_0$. Then 
\begin{equation} \label{smasympt}
\rho_n^{\pm} = n \pm a + \varkappa_n, \quad n \in \mathbb Z,
\end{equation}
and the product \eqref{prodP} can be rewritten in the form
$$
P(\la) = d^+(\rho) d^-(\rho), \quad d^{\pm}(\rho) := \left(1 - \frac{\rho}{\rho_0^{\pm}} \right)\prod_{\substack{n = -\iy \\ n \ne 0 }}^{\iy}
\left( 1 - \frac{\rho}{\rho_n^{\pm}} \right) \exp\left( \frac{\rho}{\rho_n^{\pm}}\right)
$$ 
Clearly, $d^+(\rho)$ and $d^-(\rho)$ are entire functions with the zeros $\{ \rho_n^+ \}_{n \in \mathbb Z}$ and $\{ \rho_n^- \}_{n \in \mathbb Z}$, respectively. Introduce the functions
$$
\tilde d^{\pm}(\rho) := \pi (\rho_0^{\pm} - \rho) \prod_{\substack{n = -\iy \\ n \ne 0}}^{\iy} \frac{\rho_n^{\pm} - \rho}{n} \exp\left( \frac{\rho \mp a}{\rho_n^{\pm} \mp a}\right), 
$$
having the same zeros. For simplicity, we assume that $\rho_n^{\pm} \ne \pm a$.
One can easily calculate
\begin{equation} \label{fracd}
\frac{d^{\pm}(\rho)}{\tilde d^{\pm}(\rho)} = \frac{1}{\pi \rho_0^{\pm}} \exp\left( \rho \sum_{\substack{n = -\iy \\ n \ne 0}}^{\iy} \frac{\mp a}{\rho_n^{\pm} (\rho_n^{\pm} \mp a)}\right) \prod_{\substack{n = -\iy \\ n \ne 0}}^{\iy} \frac{n}{\rho_n^{\pm}} \exp\left( \frac{\pm a}{\rho_n^{\pm} \mp a}\right).
\end{equation}
By virtue of the asymptotic formula \eqref{smasympt}, the sum and the product in \eqref{fracd} converge absolutely. The relation \eqref{fracd} yields
$$
\frac{d^+(\rho) d^-(\rho)}{\tilde d^+(\rho) \tilde d^-(\rho)} = C.
$$
By Lemma~\ref{lem:prod}
$$
\tilde d^{\pm}(\rho \pm a) = \sin \rho \pi + \int_0^{\pi} w_1^{\pm}(t) \sin \rho t \, dt + \int_0^{\pi} w_2^{\pm}(t) \cos \rho t \, dt.
$$
Consequently,
\begin{multline*}
P(\la) = d^+(\rho) d^-(\rho) = C \tilde d^+(\rho) \tilde d^-(\rho) = C \biggl( \sin (\rho - a) \pi + \int_0^{\pi} w_1^+(t) \sin (\rho - a) t \, dt \\ +  \int_0^{\pi} w_2^+(t) \cos(\rho - a) t\, dt \biggr) 
\biggl( \sin (\rho + a) \pi + \int_0^{\pi} w_1^-(t) \sin (\rho + a) t \, dt  +  \int_0^{\pi} w_2^-(t) \cos(\rho + a) t\, dt \biggr) \\
= 2 C (\cos 2 a \pi - \cos 2 \rho \pi) + F(\rho).
\end{multline*}
Clearly, $F \in B_{2, 2 \pi}$ and $F(\rho) = F(-\rho)$. Hence $F(\rho) = \displaystyle\int_0^{2 \pi} w(t) \cos \rho t \, dt$, where $w \in L_2(0, 2 \pi)$.
\end{proof}

Recall that $\{ \la_{nk} \}_{n \in \mathbb N, \, k = \overline{1, 4}}$ are eigenvalues of the boundary value problem $L$, satisfying the asymptotic relations \eqref{asymptrho}. For simplicity, assume that $\la_{nk} \ne 0$.
Summarizing the results of the previous corollaries, we obtain the following one.

\begin{cor} \label{cor:prodla}
The function
$$
P(\la) := \prod_{k = 1}^4 \prod_{n = 1}^{\iy} \left( 1 - \frac{\la}{\la_{nk}} \right)
$$
admits the representation
\begin{equation} \label{reprP}
P(\la) = \frac{C}{\rho} \sin \rho \pi \cos \rho \pi (\cos 2 \rho \pi - \cos 2 \al) + \frac{1}{\rho} \int_0^{4 \pi} w(t) \sin \rho t \, dt, \quad w \in L_2(0, 4 \pi).
\end{equation}
The following estimate from below is valid
$$
|P(\rho^2)| \ge C |\rho|^{-1} \exp(4 |\mbox{Im}\, \rho| \pi), \quad \eps < \arg \rho < \pi - \eps, \quad
|\rho| \ge \rho^*,
$$
for some positive $\eps$ and $\rho^*$.
\end{cor}

\medskip

{\bf Acknowledgment.} This work was supported in part by the Russian Federation President Grant MK-686.2017.1, by Grant
1.1660.2017/PCh of the Russian Ministry of Education and Science and by Grants 15-01-04864, 16-01-00015, 17-51-53180 of the Russian Foundation for Basic Research.

\medskip

\noindent Natalia Pavlovna Bondarenko \\
1. Department of Applied Mathematics, Samara National Research University, \\
34, Moskovskoye Shosse, Samara 443086, Russia, 
2. Department of Mechanics and Mathematics, Saratov State University, \\
Astrakhanskaya 83, Saratov 410012, Russia, \\
e-mail: {\it BondarenkoNP@info.sgu.ru}

\end{document}